\documentclass{amsart}
\usepackage{amsmath,amssymb,amsthm,mathrsfs,mathtools}
\usepackage[colorlinks]{hyperref}
\usepackage{tikz}

\usepackage[backend=bibtex,doi=false,isbn=false,url=false,backref=true]{biblatex}
\newtheorem{dfn}{Definition}[section]
\newtheorem{thm}[dfn]{Theorem}

\newtheorem{lem}[dfn]{Lemma}
\newtheorem{cor}[dfn]{Corollary}
\newtheorem{remark}[dfn]{Remark}

\numberwithin{equation}{section}

\title[Inverse flux problems]{Stability for an inverse flux and an inverse boundary coefficient problems}

\author{Mourad Choulli}
\address{Universit\'{e} de Lorraine, 34 cours L\'{e}opold, 54052 Nancy cedex, France}
\email{mourad.choulli@univ-lorraine.fr}

\author{Shuai Lu}
\address{School of Mathematical Sciences, Fudan University, 220 Handan Road, Shanghai 200433, China}
\email{slu@fudan.edu.cn}

\author{Hiroshi Takase}
\address{Institute of Mathematics for Industry, Kyushu University, 744 Motooka, Nishi-ku, Fukuoka 819-0395, Japan}
\email{htakase@imi.kyushu-u.ac.jp}

\date{\today}
\begin{document}
\begin{abstract}
We establish both Lipschitz and logarithmic stability estimates for an inverse flux problem and subsequently apply these results to an inverse boundary coefficient problem. Furthermore, we demonstrate how the stability inequalities derived for the inverse boundary coefficient problem can be utilized in solving an inverse corrosion problem. This involves determining the unknown corrosion coefficient on an inaccessible part of the boundary based on measurements taken on the accessible part of the boundary.
\end{abstract}

\maketitle

\section{Introduction and main results}\label{se1}

Let $n \geq 2$ be an integer. Throughout this text, we adopt the Einstein summation convention for quantities involving indices. If any index appears twice in any term, once as an upper index and once as a lower index, it is implied that a summation over that index is performed from $1$ to $n$.

In the current work, we assume that all functions in Sections \ref{se1} and \ref{se2} take complex values, whereas the discussion in Section \ref{se3} is based on real-valued functions.

Let $(g_{jk}) \in W^{1,\infty}(\mathbb{R}^n; \mathbb{R}^{n \times n})$ be a symmetric matrix-valued function satisfying, for some $\theta > 0$,
\[
g_{jk}(x)\xi^j\xi^k \geq \theta|\xi|^2 \quad \text{for all } x, \xi \in \mathbb{R}^n.
\]
Note that $(g^{jk})$, the matrix inverse to $g$, is also uniformly positive definite. Recall that the Laplace-Beltrami operator associated with the metric tensor $g = g_{jk}dx^j \otimes dx^k$ is given by
\[
\Delta_g u := \frac{1}{\sqrt{|g|}}\partial_j\left(\sqrt{|g|}g^{jk}\partial_k u\right),
\]
where $|g| = \det(g)$.

Let $B\subset\mathbb{R}^n$ be a smooth bounded open set with boundary $\mathcal{S}:=\partial B$ and set $U:=\mathbb{R}^n\setminus\overline{B}$. For convenience, we recall the following usual notations
\begin{align*}
&\langle X,Y\rangle=g_{jk}X^jY^k,\quad X=X^j\frac{\partial}{\partial x_j},\; Y=Y^j\frac{\partial}{\partial  x_j},
\\
&\nabla_g w=g^{jk}\partial_j w\frac{\partial}{\partial x_k},\quad w\in H^1(U),
\\
&|\nabla_g w|_g^2=\langle\nabla_g w,\nabla_g \overline{w}\rangle=g^{jk}\partial_j w\partial_k \overline{w}, \quad w\in H^1(U),
\\
&\nu_g=(\nu_g)^j\frac{\partial}{\partial x_j},\quad (\nu_g)^j=\frac{g^{jk}\nu_k}{\sqrt{g^{\ell m}\nu_\ell \nu_m}},
\\
&\partial_{\nu_g}w=\langle \nabla_g w,\nu_g\rangle, \quad w\in H^1(U),
\end{align*}
where $\nu$ denotes the outer unit normal vector field on $\mathcal{S}$. Additionally, we define the tangential gradient $\nabla_{\tau_g} w$ with respect to $g$ by
\[
\nabla_{\tau_g} w := \nabla_g w - (\partial_{\nu_g} w)\nu_g.
\]
We verify that $|\nabla_{\tau_g} w|_g^2=|\nabla_g w|_g^2-|\partial_{\nu_g} w|^2$. 


Next, consider for $t \in (0,\infty)\setminus\mathbb{N}$ the space
\[
H^t(\mathbb{R}^{n-1}) := \{ w \in \mathscr{S}'; \; (1 + |\xi|^2)^{t/2}\widehat{w} \in L^2(\mathbb{R}^{n-1}) \}
\]
equipped with the norm 
\[
\| w \|_{H^t(\mathbb{R}^{n-1})} := \left\| (1 + |\xi|^2)^{t/2}\widehat{w} \right\|_{L^2(\mathbb{R}^{n-1})}, 
\]
where $\mathscr{S}'$ denotes the space of tempered distributions on $\mathbb{R}^{n-1}$ and $\widehat{w}$ represents the Fourier transform of $w$. Utilizing local charts and a partition of unity, we construct $H^t(\mathcal{S})$ from $H^t(\mathbb{R}^{n-1})$. For further details, readers are referred to \cite[Section 2.2]{Agranovich2015}.
It follows from \cite[Theorem 1.9.2 and Theorem 2.3.6]{Agranovich2015} that  the multiplication by $q \in C^{\lfloor t\rfloor,1}(\mathcal{S})$, where $\lfloor t\rfloor$ denotes the largest integer not exceeding $t$, defines a bounded operator on $H^t(\mathcal{S})$ with
\begin{equation}\label{boundedness}
\| qu \|_{H^t(\mathcal{S})} \leq \mathbf{c}_0 \| q \|_{C^{\lfloor t\rfloor,1}(\mathcal{S})} \| u \|_{H^t(\mathcal{S})},
\end{equation}
where $\mathbf{c}_0 = \mathbf{c}_0(n,t,\mathcal{S})>0$ is a constant.

Assume that $p\in L^\infty(U)$ and $q\in C^{0,1}(\mathcal{S})$ such that $\Re q\ge 0$. Then we consider the exterior boundary-value problem
\begin{equation}\label{BVP}
\begin{cases}
Pu:=-\Delta_gu+pu=f\quad &\text{in}\; U,
\\
\partial_{\nu_g} u+q u=\mathfrak{a}\quad &\text{on}\; \mathcal{S},
\end{cases}
\end{equation}
where $f\in L^2(U)$ and $\mathfrak{a}\in L^2(\mathcal{S})$. According to Lemma \ref{well-posedness} below, under the assumption $\Re p\ge \delta$, where $\delta>0$ is a constant, \eqref{BVP} has a unique solution $u=u(f,\mathfrak{a})\in H^1(U)$. Furthermore, if $\mathfrak{a}\in H^{1/2}(\mathcal{S})$, then $u(f,\mathfrak{a})\in H^2(W\cap U)$ for any bounded open set $W\Supset B$ by Lemma \ref{well-posedness}.

Let $\Omega\Supset B$ be a $C^{0,1}$ bounded domain such that $\Omega\setminus\overline{B}$ is connected with boundary $\Gamma:=\partial\Omega$. We consider \eqref{BVP} for the case $f=0$ and are interested in the inverse flux problem of determining the unknown function $\mathfrak{a}$ from the knowledge of $(u(0,\mathfrak{a})_{|\Gamma},\partial_{\nu_g} u(0,\mathfrak{a})_{|\Gamma})\in H^{3/2}(\Gamma)\times H^{1/2}(\Gamma)$, where $\nu_g$ denotes the outer unit normal to $\Gamma$ with respect to $g$. We refer to Figure \ref{fig:domain} for an illustration of the domain related to the exterior boundary-value problem \eqref{BVP}, as well as the associated measurement setup.

\begin{figure}[htbp]
\centering
\begin{tikzpicture}
\coordinate (A) at (-4,-2.2);
\coordinate (B) at (4.2,2.2); 
\fill[gray, opacity=0.3] (A) rectangle (B);   
\draw[fill=gray, fill opacity=0.5, line width=1pt] (0.4,0) ellipse (3cm and 1.5cm); 
\draw[fill=white] (0,0) ellipse (2cm and 1cm); 

\node at (0.3,0.5) {$B$};
\node at (1.7,0.2) {$\mathcal{S}$};
\node at (2.5,0.5) {$\Omega$};    
\node at (3.2,0.2) {$\Gamma$};
\node at (-2,1.2) {$U$};

\draw[->] (-3.5,0) -- (4,0); 
\draw[->] (0,-2) -- (0,2); 
\foreach \x in {-2,0,2}
\draw (\x,-0.1) -- (\x,0.1); 
\foreach \y in {-1,1}
\draw (-0.1,\y) -- (0.1,\y); 
\end{tikzpicture}
\caption{Illustration of the domain for Equation \eqref{BVP}. $\mathcal{S}:=\partial B$, $B \Subset \Omega$, $\Gamma:=\partial \Omega$ and  $U:=\mathbb{R}^n\setminus\overline{B}$. $\Gamma$ is an accessible and $\mathcal{S}$ is an inaccessible part of the boundary. }
\label{fig:domain}
\end{figure}
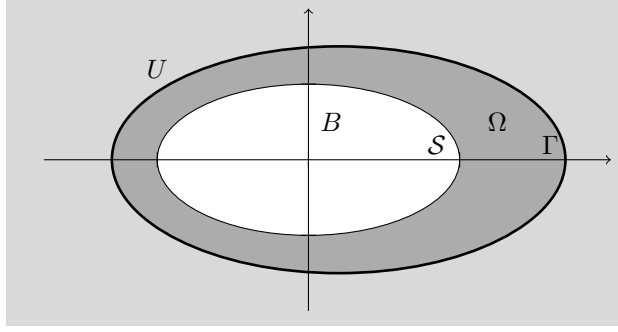

The volume form in $U$ with respect to the metric $g$ will be denoted by $dV_g := \sqrt{|g|} \, dx$, and the surface measure on $\mathcal{S} \cup \Gamma$ with respect to the induced metric from $g$ will be denoted by $dS_g$. We define the functional
\[
\mathcal{C}(w) := \|w\|_{H^1(\Gamma)} + \|\partial_{\nu_g} w\|_{L^2(\Gamma)}, \quad w \in H^2(\Omega\setminus\overline{B}),
\]
where the norms $\|w\|_{L^2(\Gamma)}$ and $\|w\|_{H^1(\Gamma)}$ are defined as follows:
\begin{align*}
\|w\|_{L^2(\Gamma)} &:= \left( \int_\Gamma |w|^2 \, dS_g \right)^{1/2}, \\
\|w\|_{H^1(\Gamma)} &:= \left( \int_\Gamma |w|^2 \, dS_g + \int_\Gamma |\nabla_{\tau_g} w|^2 \, dS_g \right)^{1/2}.
\end{align*}
For simplicity, we use hereinafter the notation $\zeta:=(g,p,q,B,\Omega)$.

We aim to prove the following two theorems. 

\begin{thm}\label{thm_single_log_ifp}
Let $0<\eta<1/4$. Then there exist constants $\mathbf{c}=\mathbf{c}(\zeta,\eta)>0$ and $c=c(\zeta,\eta)>0$ such that for all $\mathfrak{a}\in H^{1/2}(\mathcal{S})$ and $s\ge 1$ there holds
\[
\mathbf{c}\|\mathfrak{a}\|_{L^2(\mathcal{S})}\le e^{cs}\mathcal{C}(u(0,\mathfrak{a}))+s^{-\eta}\|\mathfrak{a}\|_{H^{1/2}(\mathcal{S})}.
\]
\end{thm}

Let
\[
\Phi_{\eta,c} (r)=r^{-1}\chi_{]0,e^c]}(r)+(\log r)^{-\eta}\chi_{]e^c,\infty[}(r),\quad 0<\eta<1/4,\; c>0.
\]
Here $\chi_J$ denotes the characteristic function of the interval $J$.

By minimizing the interpolation inequality of Theorem \ref{thm_single_log_ifp} with respect to $s \geq 1$, we derive the following corollary.
\begin{cor}
Let $0<\eta<1/4$. There exist constants $\mathbf{c}=\mathbf{c}(\zeta,\eta)>0$ and $c=c(\zeta,\eta)>0$ such that for all $\mathfrak{a}\in H^{1/2}(\mathcal{S})\setminus\{0\}$  there holds the inequality
\[
\mathbf{c}\|\mathfrak{a}\|_{L^2(\mathcal{S})}\le \Phi_{\eta,c} \left(\|\mathfrak{a}\|_{H^{1/2}(\mathcal{S})}/\mathcal{C}(u(0,\mathfrak{a}))\right)\|\mathfrak{a}\|_{H^{1/2}(\mathcal{S})}.
\]
\end{cor}

If we enhance the {\it a priori} assumption of the unknown flux function $\mathfrak{a}$, the aforementioned logarithmic stability estimate can be upgraded to a Lipschitz stability estimate. To achieve this, we fix $M > 0$ and define the set
\[
\mathcal{A}=\{\mathfrak{a}\in H^1(\mathcal{S});\; \|\nabla_{\tau_g}\mathfrak{a}\|_{L^2(\mathcal{S})}\le M\|\mathfrak{a}\|_{L^2(\mathcal{S})}\},
\]
where the norms $\|w\|_{L^2(\mathcal{S})}$ and $\|\nabla_{\tau_g}w\|_{L^2(\mathcal{S})}$ are defined similarly to $\|w\|_{L^2(\Gamma)}$. Roughly speaking, $\mathcal{A}$ imposes a constraint on the high-frequency components. Indeed, as shown in Corollary \ref{cor_finite_dim}, finite-dimensional subspaces are contained in $\mathcal{A}$.

In a subsequent theorem, we demonstrate the Lipschitz stability estimate when $\mathfrak{a}$ belongs to the set $\mathcal{A}$.

\begin{thm}\label{thm_Lipschitz_ifp} 
There exists a constant $\mathbf{c} = \mathbf{c}(\zeta, M)>0$ such that for any $\mathfrak{a}\in\mathcal{A}$ the following inequality is satisfied:
\[
\mathbf{c}\|\mathfrak{a}\|_{H^1(\mathcal{S})}\le \mathcal{C}(u(0,\mathfrak{a})).
\]
\end{thm}

\begin{remark}{\rm
The same results are valid for the interior case as well. More specifically, let $B$ be a smooth bounded domain and $\Omega \Subset B$ be a $C^{0,1}$ bounded open subset such that $B \setminus \overline{\Omega}$ is connected. When measurements are taken on $\Gamma := \partial \Omega$, we can establish the same stability inequalities as those presented in Theorem \ref{thm_single_log_ifp} and Theorem \ref{thm_Lipschitz_ifp}. The interior problem of this nature has been addressed in Choulli-Takase \cite{Choulli2025}.
}\end{remark}

To the best of our knowledge, these results provide stability inequalities for the inverse problems we consider for the first time. Both Theorems \ref{thm_single_log_ifp} and \ref{thm_Lipschitz_ifp} remain valid by replacing the operator $P$, appearing in \eqref{BVP}, by $P$ plus a first-order operator, provided that the corresponding boundary-value problems admit a unique solution in $H^1(U)\cap H^2(W\cap U)$ for any bounded open set $W\Supset B$.

\section{Proof of the main results}\label{se2}
In this section, we aim to prove the aforementioned two primary theorems. In the process, we will require the following lemma, which will be utilized subsequently.

\begin{lem}\label{well-posedness}
Let $p\in L^\infty(U)$ and $q\in C^{0,1}(\mathcal{S})$ satisfying $\Re p\ge \delta$ for some $\delta>0$ and $\Re q\ge 0$. If $(f,\mathfrak{a})\in L^2(U)\times L^2(\mathcal{S})$, then there exists a unique solution $u:=u(f,\mathfrak{a})\in H^1(U)$ to \eqref{BVP} and there holds the estimate
\begin{equation}\label{H^1_estimate}
\mathbf{c}_1\|u\|_{H^1(U)}\le \left(\|f\|_{L^2(U)}+\|\mathfrak{a}\|_{L^2(\mathcal{S})}\right),
\end{equation}
where $\mathbf{c}_1=\mathbf{c}_1(g,p,B)>0$ is a constant.

If in addition to the above assumptions, $(g_{\ell m})\in W^{k+1,\infty}(\mathbb{R}^n; \mathbb{R}^{n\times n})$, 
$p\in W^{k,\infty}(U)$, $q\in C^{k,1}(\mathcal{S})$ and $(f,\mathfrak{a})\in H^k(U)\times H^{k+1/2}(\mathcal{S})$ for $k\in\mathbb{N}\cup\{0\}$, then $u\in H^{k+2}(W\cap U)$ for any bounded open set $W\Supset B$ and we have
\begin{equation}\label{H^k_estimate}
\mathbf{c}_2\|u\|_{H^{k+2}(W\cap U)}\le \left(\|f\|_{H^k(U)}+\|\mathfrak{a}\|_{H^{k+1/2}(\mathcal{S})}\right),
\end{equation}
where $\mathbf{c}_2=\mathbf{c}_2(g,p,q,k,B,W)>0$ is a constant.
\end{lem}

\begin{proof}
When $\Re p\ge \delta>0$, the sesquilinear form
\[
h(u,v):=\int_U(\langle\nabla_g u,\nabla_g\overline{v}\rangle+p u\overline{v})dV_g+\int_\mathcal{S}qu\overline{v}dS_g,\quad u,v\in H^1(U)
\]
is bounded and coercive on $H^1(U)$. Consequently, for all $(f,\mathfrak{a})\in L^2(U)\times L^2(\mathcal{S})$, Lax-Milgram lemma asserts the existence of a unique solution $u \in H^1(U)$ such that
\begin{equation}\label{VP}
h(u,v)=\int_U f\overline{v}dV_g+\int_\mathcal{S}\mathfrak{a}\overline{v}dS_g,\quad v\in H^1(U).
\end{equation}
By taking the real part of $h(u,u)$, we derive
\[
\mathbf{c}_1\|u\|_{H^1(U)}^2\le \|f\|_{L^2(U)}\|u\|_{L^2(U)}+\|\mathfrak{a}\|_{L^2(\mathcal{S})}\|u\|_{L^2(\mathcal{S})},
\]
where $\mathbf{c}_1=\mathbf{c}_1(g,p,B)>0$ is a constant. Since the trace mapping $H^1(U)\ni u\mapsto u_{|\mathcal{S}}\in L^2(\mathcal{S})$ is bounded, we obtain \eqref{H^1_estimate}.

Next, we provide the proof of \eqref{H^k_estimate} in the specific case where $k=0$. The general case can be derived by induction with respect to $k$. For a bounded open set $W\Supset B$, we choose a $C^{1,1}$ bounded open set $W_0\Supset W$. Let $\chi \in C_0^\infty (W_0)$ satisfying $\chi=1$ in a neighborhood of $W$. According to \cite[Theorem 2.4.2.7]{Grisvard1985}, we have $\chi u\in H^2(W_0\cap U)$, where $u$ denotes the solution to \eqref{VP}. By applying \cite[Theorem 2.3.3.6]{Grisvard1985} to $\chi u$ in $W_0\cap U$ and utilizing \eqref{H^1_estimate}, we obtain
\begin{align*}
\|u\|_{H^2(W\cap U)}&\le\|\chi u\|_{H^2(W_0\cap U)}
\\
&\le \mathbf{c}_2\left(\|f\|_{L^2(U)}+\|u\|_{H^1(U)}+\|\partial_{\nu_g}u\|_{H^{1/2}(\mathcal{S})}\right)
\\
&\le \mathbf{c}_2\left(\|f\|_{L^2(U)}+\|\mathfrak{a}\|_{H^{1/2}(\mathcal{S})}+\mathbf{c}_0\|q\|_{C^{0,1}(\mathcal{S})}\|u\|_{H^{1/2}(\mathcal{S})}\right)
\\
& \le \mathbf{c}_2\left(\|f\|_{L^2(U)}+\|\mathfrak{a}\|_{H^{1/2}(\mathcal{S})}\right),
\end{align*}
where $\mathbf{c}_2=\mathbf{c}_2(g,p,q,B,W)>0$ is a generic constant.
\end{proof}

We are now ready to prove Theorems \ref{thm_single_log_ifp} and \ref{thm_Lipschitz_ifp}, respectively. 

\begin{proof}[Proof of Theorem \ref{thm_single_log_ifp}]

Let $D := \Omega \setminus \overline{B}$, $\mathfrak{a} \in H^{1/2}(\mathcal{S})$ and $u = u(0,\mathfrak{a}) \in H^2(D)$. According to \cite[Theorem 1.1]{CT2024b}, for $3/2 < \eta < 2$, there exist constants $\mathbf{c} = \mathbf{c}(g, p, B, \Omega, \eta) > 0$ and $c = c(g, p, B, \Omega, \eta) > 0$ such that for all $s \geq 1$, we have the inequality
\[
\mathbf{c}\|u\|_{H^\eta (D)}\le e^{cs}\mathcal{C}(u)+s^{-(2-\eta)/2}\|u\|_{H^2(D)}.
\]

Henceforth, $\mathbf{c}=\mathbf{c}(g,p,B,\Omega,\eta)>0$ will denote a generic constant. Since the trace mapping
\[
H^\eta(D)\ni u\mapsto (u_{|\mathcal{S}},\partial_{\nu_g}u_{|\mathcal{S}})\in L^2(\mathcal{S})\times L^2(\mathcal{S})
\]
is bounded, we can derive the following inequality:
\[
\mathbf{c}\left(\|u\|_{L^2(\mathcal{S})}+\|\partial_{\nu_g}u\|_{L^2(\mathcal{S})}\right)\le e^{cs}\mathcal{C}(u)+s^{-(2-\eta)/2}\|u\|_{H^2(D)},\quad s\ge 1.
\]

By using $\mathfrak{a}=\partial_{\nu_g}u+qu$ and applying \eqref{H^k_estimate} with $k = 0$, we obtain the following inequality:
\[
\mathbf{c}\|\mathfrak{a}\|_{L^2(\mathcal{S})}\le e^{cs}\mathcal{C}(u)+s^{-(2-\eta)/2}\|\mathfrak{a}\|_{H^{1/2}(\mathcal{S})},\quad s\ge 1,
\]
where $\mathbf{c}$ depends also on $\|q\|_{L^\infty(\mathcal{S})}$. By replacing $(2-\eta)/2$ with $\eta$, the expected inequality follows.
\end{proof}

\begin{proof}[Proof of Theorem \ref{thm_Lipschitz_ifp}]
By Theorem \ref{thm_single_log_ifp} with $\eta := 1/8$, we deduce that
\begin{align*}
\mathbf{c}\|\mathfrak{a}\|_{L^2(\mathcal{S})}&\le e^{cs}\mathcal{C}(u)+s^{-1/8}\|\mathfrak{a}\|_{H^{1/2}(\mathcal{S})}
\\
&\le e^{cs}\mathcal{C}(u)+s^{-1/8}\|\mathfrak{a}\|_{H^1(\mathcal{S})},\quad s\ge 1,
\end{align*}
where $\mathbf{c}(\zeta) > 0$ and $c(\zeta) > 0$ are constants as specified in Theorem \ref{thm_single_log_ifp}. Given that $\mathfrak{a} \in \mathcal{A}$, we observe that
\begin{align*}
\|\mathfrak{a}\|_{H^1(\mathcal{S})}&\le \|\mathfrak{a}\|_{L^2(\mathcal{S})}+\|\nabla_{\tau_g}\mathfrak{a}\|_{L^2(\mathcal{S})}
\\
&\le (1+M)\|\mathfrak{a}\|_{L^2(\mathcal{S})}.
\end{align*}
For sufficiently large $s \geq 1$, we can assume that $\mathbf{c}(\zeta) - (1 + M)s^{-1/8} \geq \mathbf{c}(\zeta)/2$, allowing us to rearrange the inequality as
\[
\mathbf{c}\|\mathfrak{a}\|_{L^2(\mathcal{S})}\le e^{cs}\mathcal{C}(u),\quad s\ge 1,
\]
by modifying the constant $\mathbf{c}$ depending also on $M$. The proof completed by using again $\mathfrak{a}\in\mathcal{A}$.
\end{proof}

\section{Inverse boundary coefficient problem}\label{se3}
In this section, we delve into an application of the aforementioned Theorems \ref{thm_single_log_ifp} and \ref{thm_Lipschitz_ifp} in the context of an inverse boundary coefficient problem. Throughout the subsequent discussion, it is assumed that all functions are real-valued for simplicity. The reason is that, while a parallel argument can be made for complex-valued functions by decomposing them into their real and imaginary parts, this leads to rather cumbersome notation.

The geometric setup for the problem we intend to investigate is rooted in the description provided in \eqref{BVP}, but it is considered as an interior boundary value problem. For clarity, we recall the relevant details below.
Let $n\ge 2$ be an integer, $B\subset\mathbb{R}^n$ be a smooth bounded open set and $\Omega\Supset B$ be a smooth bounded domain such that $D:=\Omega\setminus\overline{B}$ is connected. Set $\mathcal{S}:=\partial B$, $\Gamma:=\partial \Omega$ so that $\partial D=\mathcal{S}\cup \Gamma$ and define the set
\[
\mathscr{Q}_k:=\{q\in C^{k,1}(\partial D);\; q\ge 0\; \text{and}\; q\not\equiv 0\},\quad k\in\mathbb{N}\cup\{0\}.
\]

Consider the boundary-value problem
\begin{equation}\label{BVP_2}
\begin{cases}
-\Delta_g u=f\quad &\text{in}\; D,
\\
\partial_{\nu_g} u+q u=\mathfrak{a}\quad &\text{on}\; \partial D,
\end{cases}
\end{equation}
where $f\in L^2(D)$, $q\in\mathscr{Q}_0$ and $\mathfrak{a}\in H^{1/2}(\partial D)$.

In the case where $\mathrm{supp}(\mathfrak{a}) \subset \Gamma$ and $\mathrm{supp}(q) \subset \mathcal{S}$, the boundary-value problem \eqref{BVP_2} can be considered as a standard model in non-destructive testing. Here, the coefficient $q$ represents the characteristic of the inaccessible part of the boundary, such as regions where corrosion may have occurred. In practical applications, measurements are typically conducted on $\Gamma$, which constitutes the accessible portion of the boundary.

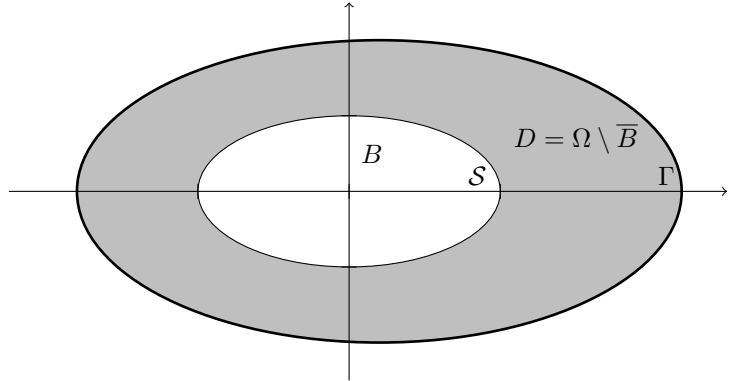
\begin{figure}[htbp]
\centering
\begin{tikzpicture}
\draw[fill=gray, fill opacity=0.5, line width=1pt] (0.4,0) ellipse (4cm and 2cm); 
\draw[fill=white ] (0,0) ellipse (2cm and 1cm); 

\node at (0.3,0.5) {$B$};
\node at (1.7,0.2) {$\mathcal{S}$};
\node at (3,0.7) {$D=\Omega\setminus \overline{B}$};    
\node at (4.2,0.2) {$\Gamma$};

\draw[->] (-4.5,0) -- (5,0); 
\draw[->] (0,-2.5) -- (0,2.5); 
\foreach \x in {-2,0,2}
\draw (\x,-0.1) -- (\x,0.1); 
\foreach \y in {-1,1}
\draw (-0.1,\y) -- (0.1,\y); 
\end{tikzpicture}
\caption{Illustration of the domain for Equation \eqref{BVP_2}. $B\subset\mathbb{R}^n$, $\Omega\Supset B$ and $D:=\Omega\setminus\overline{B}$. The boundaries are $\mathcal{S}:=\partial B$ and
$\Gamma:=\partial \Omega$ so that $\partial D=\mathcal{S}\cup \Gamma$.}
\end{figure}

Using an argument similar to that in the proof of Lemma \ref{well-posedness}, and considering the equivalent scalar product $h$ defined in $H^1(D)$ by
\[
h(u,v):=\int_D \langle\nabla_g u,\nabla_g v\rangle dV_g+\int_{\partial D}quvdS_g,\quad u,v\in H^1(D),
\]
(see, for instance, \cite[Lemma 2.9]{Choulli2016} for a similar approach), we can conclude that the boundary-value problem \eqref{BVP_2} admits a unique solution $u := u_q(f, \mathfrak{a}) \in H^2(D)$.

Fix an arbitrary $\kappa > 0$ and define the set
\[
\mathscr{Q}:=\{q\in\mathscr{Q}_k;\; \|q\|_{C^{k,1}(\partial D)}\le  \kappa\}.
\]
The inverse coefficient problem tackled in this section aims to reconstruct the restriction $q_{|\mathcal{S}}$ of the unknown function $q$, belonging to the set $\mathscr{Q}$, near the boundary $\mathcal{S}$. This reconstruction is based on the knowledge of the boundary data
\[
(u_q(f,\mathfrak{a})_{|\Gamma},\partial_{\nu_g} u_q(f,\mathfrak{a})_{|\Gamma})\in H^{3/2}(\Gamma)\times H^{1/2}(\Gamma).
\]

There are numerous stability results associated to this inverse problem. For logarithmic stability estimates, encompassing single, double, and triple logarithms, readers are directed to the following references: \cite{Alessandrini2003, Chaabane2004, Cheng2008, Choulli2018, Sincich2015, Choulli2016a, Choulli2016}. For stronger stability estimates, Choulli \cite[Theorem 1.2]{Choulli2002} establishes a H\"{o}lder type stability estimate when observation data are collected on the same boundary as the unknown function. Additionally, Hu and Yamamoto \cite{Hu2015} demonstrate a H\"{o}lder type stability estimate in the two-dimensional case when the unknown functions are constrained to be analytic or piecewise constant. Chaabane and Jaoua \cite{Chaabane1999} establish a Lipschitz type stability estimate under a monotony assumption concerning the unknown coefficients, while Sincich \cite{Sincich2007} proves another Lipschitz type stability estimate when the unknown coefficients are assumed to be piecewise constant functions. Furthermore, Yang, Choulli, and Cheng \cite{Yang2005} propose a numerical scheme employing the boundary-element method for the specific scenario where $D$ is an annulus.

In this section, we first apply Theorem \ref{thm_Lipschitz_ifp} to derive a Lipschitz stability estimate for the aforementioned inverse coefficient problem, provided that the coefficient $q$ fulfills specific constraints. For clarity and consistency, we reuse the notation defined as follows:
\[
\mathcal{C}(w)=\|w\|_{H^1(\Gamma)}+\|\partial_{\nu_g} w\|_{L^2(\Gamma)},\quad w\in H^2(D)
\]
and we denote $\varsigma:=(g,\kappa,k,B,\Omega,f,\mathfrak{a})$.

\begin{thm}\label{thm_Lipschitz_iccp}
Let $k>n/2$ be an integer and $K>0$. Assume that 
$(g_{\ell m})\in W^{k+1,\infty}(\mathbb{R}^n;\mathbb{R}^{n\times n})$. Fix $(f,\mathfrak{a})\in H^k(D)\times H^{k+1/2}(\partial D)\setminus\{(0,0)\}$ such that  $f\ge 0$ and $\mathfrak{a}\ge 0$. Then, there exists a constant $\mathbf{c}=\mathbf{c}(\varsigma,K)>0$ such that for all $q_j\in\mathscr{Q}$ and $u_j=u_{q_j}(f,\mathfrak{a})\in H^{k+2}(D)$ satisfying \eqref{BVP_2} for $j=1,2$, there holds
\[
\mathbf{c}\|q_1-q_2\|_{H^1(\mathcal{S})}\le \mathcal{C}(u_1-u_2),
\]
provided that the coefficients $q_1, q_2$ satisfy the condition 
\begin{equation}\label{constraint}
\|\nabla_{\tau_g}(q_1-q_2)\|_{L^2(\mathcal{S})}\le K\|q_1-q_2\|_{L^2(\mathcal{S})}.
\end{equation}
\end{thm}

Before proving Theorem \ref{thm_Lipschitz_iccp}, we present some of its consequences. Let $(\lambda_m)$ be the sequence of eigenvalues of the Laplace-Beltrami operator $-\Delta_{\tau_g}$ on $\mathcal{S}$, satisfying
\[
0 = \lambda_0 \leq \lambda_1 \leq \lambda_2 \leq \cdots \quad \text{and} \quad \lim_{m \to \infty} \lambda_m = \infty.
\]
Let $(\phi_m)$ be an orthonormal basis of $L^2(\mathcal{S})$ consisting of eigenfunctions, where each $\phi_m$ corresponds to the eigenvalue $\lambda_m$. For all $\lambda > 0$, we define
\[
W_\lambda = \mathrm{span}\{\phi_m;\; \lambda_m \leq \lambda\}.
\]

\begin{cor}\label{cor_finite_dim}
Let $\lambda >0$ be fixed. There exists a constant $\mathbf{c}=\mathbf{c}(\varsigma,\lambda)>0$ such that for all $q_j\in \mathscr{Q}$, $u_j=u_{q_j}(f,\mathfrak{a})\in H^{k+2}(D)$ satisfying \eqref{BVP_2} for $j=1,2$ we have
\[
\mathbf{c}\|q_1-q_2\|_{H^1(\mathcal{S})}\le \mathcal{C}(u_1-u_2),
\]
provided that $q_1-q_2\in W_\lambda$.
\end{cor}

\begin{proof}
Since
\[
q_1 - q_2 = \sum_{\lambda_m \leq \lambda} a_m \phi_m,
\]
where $a_m := (q_1 - q_2, \phi_m)_{L^2(\mathcal{S})}$, we obtain
\[
\|\nabla_{\tau_g}(q_1 - q_2)\|_{L^2(\mathcal{S})} = \sqrt{\sum_{\lambda_m \leq \lambda}\lambda_m |a_m|^2} \leq \sqrt{\lambda} \|q_1 - q_2\|_{L^2(\mathcal{S})}.
\]
The proof is completed by utilizing \eqref{constraint} and Theorem \ref{thm_Lipschitz_iccp}.
\end{proof}

If $q$ is known on the boundary $\Gamma$, the previous result can be improved, as demonstrated by the following corollary.

\begin{cor}
Let $\lambda>0$ be fixed. There exists a constant $\mathbf{c} = \mathbf{c}(\varsigma, \lambda) > 0$ such that, for all $q_j \in \mathscr{Q}$ and $u_j = u_{q_j}(f, \mathfrak{a}) \in H^{k+2}(D)$ satisfying \eqref{BVP_2} for $j = 1, 2$ we have
\[
\mathbf{c} \|q_1 - q_2\|_{H^1(\mathcal{S})} \leq \|u_1 - u_2\|_{H^1(\Gamma)},
\]
provided that $q_1 - q_2 \in W_\lambda$ and $q_1 \equiv q_2$ on $\Gamma$.
\end{cor}
\begin{proof}
Since $\partial_{\nu_g}(u_1 - u_2) + q_1(u_1 - u_2) = 0$ holds on $\Gamma$, we have
\[
\mathcal{C}(u_1-u_2)\le (1+\kappa)\|u_1-u_2\|_{H^1(\Gamma)},
\]
which concludes the proof.
\end{proof}

The following lemma will be utilized in the proof of Theorem \ref{thm_Lipschitz_iccp}.

\begin{lem}
Let $k>n/2$ be an integer. Assume that 
$(g_{\ell m})\in W^{k+1,\infty}(\mathbb{R}^n;\mathbb{R}^{n\times n})$. Fix $(f,\mathfrak{a})\in H^k(D)\times H^{k+1/2}(\partial D)\setminus\{(0,0)\}$ such that $f\ge 0$ and $\mathfrak{a}\ge 0$. Then, there exist constants $\mathfrak{c}_1=\mathfrak{c}_1(\varsigma)>0$ and $\mathfrak{c}_2=\mathfrak{c}_2(\varsigma)>0$ such that for all $q\in\mathscr{Q}$ and $u_q=u_{q}(f,\mathfrak{a})\in H^{k+2}(D)$ satisfying \eqref{BVP_2}, there hold
\begin{equation}\label{upper_bound}
\|u_q\|_{C^2(\overline{D})}\le \mathfrak{c}_1
\end{equation}
and
\begin{equation}\label{lower_bound}
\min_{\overline{D}}u_q\ge \mathfrak{c}_2\; (>0).
\end{equation}
\end{lem}
\begin{proof}
Similarly to Lemma \ref{well-posedness}, we prove that \eqref{BVP_2} admits a unique solution $u_q = u_q(f, \mathfrak{a}) \in H^{k+2}(D)$ and that \eqref{H^k_estimate} holds with $W=\Omega$ by replacing $(U,\mathcal{S})$ with $(D,\partial D)$. This, combined with the fact that the embedding $H^{k+2}(D) \hookrightarrow C^2(\overline{D})$ is continuous, implies that $u_q \in C^2(\overline{D})$ and demonstrates that \eqref{upper_bound} is satisfied.

Let $u_\kappa:=u_\kappa(f,\mathfrak{a})\in C^2(\overline{D})$ be the unique solution of \eqref{BVP_2} with $q:=\kappa$. We will prove that $\min_{\overline{D}}u_\kappa>0$. In the case where $u_\kappa$ is equal to a constant, it is nonzero by $(f,\mathfrak{a})\neq (0,0)$. Therefore, from $\kappa u_\kappa=\mathfrak{a}\ge 0$ on $\partial D$, we deduce that $\mathfrak{a}$ and $u_\kappa=\mathfrak{a}/\kappa>0$ are both equal to positive constants. When $u_\kappa$ is not equal to a constant, $u_\kappa$ reaches its minimum at some $x_\kappa \in \partial D$ according to the strong maximum principle (e.g., \cite[Theorem 3.5]{Gilbarg2001}). We have $u_\kappa(x_\kappa)>0$ because otherwise we will have
\[
\partial_{\nu_g}u_\kappa(x_\kappa)=-\kappa u_\kappa(x_\kappa)+\mathfrak{a}(x_\kappa)\ge 0,
\]
which contradicts $\partial_{\nu_g}u_\kappa(x_\kappa)<0$ by \cite[Lemma 3.1]{Ammari2020}.

Next, we verify that $v:=u_q-u_\kappa$ satisfies
\[
\begin{cases}
\Delta_g v=0\quad &\text{in}\; D,
\\
\partial_{\nu_g} v+q v=(\kappa-q)u_\kappa\quad &\text{on}\; \partial D.
\end{cases}
\]
Since $(\kappa-q)u_\kappa\ge \left(\kappa-\|q\|_{C(\partial D)}\right)u_\kappa(x_\kappa)\ge 0$ on $\partial D$, proceeding as above, we obtain $v\ge 0$ in $\overline{D}$. Thus, $\min_{\overline{D}}u_q\ge\min_{\overline{D}}u_\kappa>0$, which means that \eqref{lower_bound} holds with $\mathfrak{c}_2:=\min_{\overline{D}}u_\kappa>0$.
\end{proof}

With the above established lemma, we are now ready to proceed with the proof of Theorem \ref{thm_Lipschitz_iccp}.

\begin{proof}[Proof of Theorem \ref{thm_Lipschitz_iccp}]
Let $q_j\in\mathscr{Q}$ and $u_j=u_{q_j}(f,\mathfrak{a})$ for $j=1,2$. Then, the function $v:=u_1-u_2$ satisfies the system
\[
\begin{cases}
\Delta_g v=0\quad &\text{in}\; D,
\\
\partial_{\nu_g} v+q_1 v=\mathfrak{b}\quad &\text{on}\; \partial D,
\end{cases}
\]
where $\mathfrak{b}:=(q_2-q_1)u_2$. To apply Theorem \ref{thm_Lipschitz_ifp}, we will check that $\mathfrak{b}\in\mathcal{A}$ for some $M>0$. Using \eqref{upper_bound} and \eqref{constraint}, we derive the following estimates
\begin{align*}
\|\nabla_{\tau_g}\mathfrak{b}\|_{L^2(\mathcal{S})}&\le \mathfrak{c}_1\left(\|q_2-q_1\|_{L^2(\mathcal{S})}+\|\nabla_{\tau_g}(q_2-q_1)\|_{L^2(\mathcal{S})}\right)
\\
&\le \mathfrak{c}_1(1+K)\|q_1-q_2\|_{L^2(\mathcal{S})},
\end{align*}
where $\mathfrak{c}_1=\mathfrak{c}_1(\varsigma)>0$ is the constant as in \eqref{upper_bound}. 

On the other hand, from \eqref{lower_bound}, there exists a constant $\mathfrak{c}_2 = \mathfrak{c}_2(\varsigma) > 0$ such that $\min_{\overline{D}} u_2 \geq \mathfrak{c}_2$. Furthermore, we have
\begin{equation}\label{q_1-q_2_ineq}
\|\mathfrak{b}\|_{L^2(\mathcal{S})} \geq \mathfrak{c}_2 \|q_1 - q_2\|_{L^2(\mathcal{S})}.
\end{equation}
Combining these inequalities, we obtain
\[
\|\nabla_{\tau_g} \mathfrak{b}\|_{L^2(\mathcal{S})} \leq \mathfrak{c}_1 \mathfrak{c}_2^{-1} (1 + K) \|\mathfrak{b}\|_{L^2(\mathcal{S})},
\]
which implies that $\mathfrak{b} \in \mathcal{A}$ with $M := \mathfrak{c}_1 \mathfrak{c}_2^{-1} (1 + K) > 0$. Therefore, Theorem \ref{thm_Lipschitz_ifp} implies that
\[
\mathbf{c} \|\mathfrak{b}\|_{H^1(\mathcal{S})} \leq \mathcal{C}(v)
\]
for some constant $\mathbf{c} = \mathbf{c}(\varsigma, K) > 0$. By adjusting the constant $\mathbf{c}$, the inequality above and \eqref{q_1-q_2_ineq} yield
\[
\mathbf{c} \|q_1 - q_2\|_{L^2(\mathcal{S})} \leq \mathcal{C}(v).
\]
Finally, the proof is completed by applying \eqref{constraint} again.
\end{proof}

Finally, let us demonstrate that, despite the absence of the necessary assumption \eqref{constraint}, we can nevertheless establish a logarithmic-type stability estimate.

\begin{thm}\label{thm_log_iccp}
Let $k>n/2$ be an integer and $0<\eta<1/4$.  Assume that 
$(g_{\ell m})\in W^{k+1,\infty}(\mathbb{R}^n;\mathbb{R}^{n\times n})$. Fix $(f,\mathfrak{a})\in H^k(D)\times H^{k+1/2}(\partial D)\setminus\{(0,0)\}$ such that $f\ge 0$ and $\mathfrak{a}\ge 0$. Then there exist constants $\mathbf{c}=\mathbf{c}(\varsigma,\eta)>0$ and $c=c(g,\kappa,B,\Omega,\eta)>0$ such that for all $q_j\in\mathscr{Q}$ with $u_j=u_{q_j}(f,\mathfrak{a})\in H^{k+2}(D)$ satisfying \eqref{BVP_2} for $j=1,2$,  and $s\ge 1$ there holds
\[
\mathbf{c}\|q_1-q_2\|_{L^2(\mathcal{S})}\le e^{cs}\mathcal{C}(u_1-u_2)+s^{-\eta}.
\]
\end{thm}
\begin{proof}
Let $q_j \in \mathscr{Q}$ and $u_j = u_{q_j}(f, \mathfrak{a})$ for $j = 1, 2$. Then, the function $v := u_1 - u_2$ is the solution to the boundary-value problem
\[
\begin{cases}
\Delta_g v = 0 & \text{in} \; D, \\
\partial_{\nu_g} v + q_1 v = \mathfrak{b} & \text{on} \; \partial D,
\end{cases}
\]
where $\mathfrak{b} := (q_2 - q_1)u_2$. By applying Theorem \ref{thm_single_log_ifp}, we obtain the inequality
\[
\mathbf{c} \|\mathfrak{b}\|_{L^2(\mathcal{S})} \leq e^{cs} \mathcal{C}(v) + s^{-\eta} \|\mathfrak{b}\|_{H^{1/2}(\mathcal{S})}, \quad s \geq 1,
\]
where $\mathbf{c} = \mathbf{c}(g, \kappa, B, \Omega, \eta) > 0$ and $c = c(g, \kappa, B, \Omega, \eta) > 0$ are the constants defined in Theorem \ref{thm_single_log_ifp}. Combining this inequality with \eqref{q_1-q_2_ineq}, \eqref{boundedness}, and \eqref{upper_bound}, we can derive the desired inequality by possibly adjusting the constant $\mathbf{c} = \mathbf{c}(\varsigma, \eta) > 0$ if necessary.
\end{proof}

Implementing a similar approach as before, we minimize the interpolation inequality stated in Theorem \ref{thm_log_iccp} with respect to $s \geq 1$ to derive the following corollary. Recall the definition of the function:
\[
\Phi_{\eta,c} (r)=r^{-1}\chi_{]0,e^c]}(r)+(\log r)^{-\eta}\chi_{]e^c,\infty[}(r),\quad 0<\eta<1/4,\; c>0.
\]

\begin{cor} 
Let the assumptions stated in Theorem \ref{thm_log_iccp} be satisfied. Then, there exist constants $\mathbf{c} = \mathbf{c}(\varsigma, \eta) > 0$ and $c = c(g, \kappa, B, \Omega, \eta) > 0$ such that for all $q_j \in \mathscr{Q}$ and $u_j = u_{q_j}(f, \mathfrak{a}) \in H^{k+2}(D)$ satisfying \eqref{BVP_2} for $j = 1, 2$, there holds the inequality:
\[
\mathbf{c} \|q_1 - q_2\|_{L^2(\mathcal{S})} \leq \Phi_{\eta, c} \left( \frac{1}{\mathcal{C}(u_1 - u_2)} \right).
\]
\end{cor}

\section*{Acknowledgement}
This work was supported by National Key Research and Development Programs of China (No. 2023YFA1009103), JSPS KAKENHI (Grant Numbers JP23KK0049 and JP25K17280), NSFC (No. 11925104) and Science and Technology Commission of Shanghai Municipality (No. 23JC1400501).

\section*{Declarations}
\subsection*{Conflict of interest}
The authors declare that they have no conflict of interest.

\subsection*{Data availability}
Data sharing not applicable to this article as no datasets were generated or analyzed during the current study.

\printbibliography

\end{document}